\documentclass[12pt]{amsart}
\makeatletter
\let\@@pmod\pmod
\DeclareRobustCommand{\pmod}{\@ifstar\@pmods\@@pmod}
\def\@pmods#1{\mkern4mu({\operator@font mod}\mkern 6mu#1)}
\makeatother
\usepackage{fullpage,mathtools,amssymb}
\usepackage{hyperref}
\newtheorem{theorem}{Theorem}
\newtheorem{lemma}[theorem]{Lemma}
\numberwithin{theorem}{section}
\numberwithin{equation}{section}
\newcommand{\C}{\mathbb{C}}
\newcommand{\A}{\mathbb{A}}
\newcommand{\Q}{\mathbb{Q}}
\newcommand{\R}{\mathbb{R}}

\renewcommand{\H}{\mathbb{H}}
\DeclareMathOperator{\GL}{GL}
\begin{document}
\title{A converse theorem without root numbers}
\author{Andrew R. Booker}
\thanks{The author was partially supported by EPSRC Grant {\tt EP/K034383/1}.}
\address{
Howard House\\
University of Bristol\\
Queens Ave\\
Bristol\\
BS8 1SN\\
United Kingdom
}
\email{\tt andrew.booker@bristol.ac.uk}
\begin{abstract}
We answer a challenge posed in \cite[\S1.3]{booker} by proving a version
of Weil's converse theorem \cite{weil} that assumes a functional equation
for character twists but allows their root numbers to vary arbitrarily.
\end{abstract}
\maketitle

\section{Introduction}
When Weil introduced his converse theorem \cite{weil}, he had in
mind what eventually became known as the Shimura--Taniyama conjecture
connecting elliptic curves over $\Q$ with classical modular forms. Soon
after, Weil's theorem was recast in representation-theoretic terms by
Jacquet and Langlands \cite[Theorem~11.3]{JL}, for whom the motivation
was Artin's conjecture, now seen as the prototypical case of Langlands'
functoriality \cite[\S12]{JL}. At the time, much more was known about
the analytic properties of Artin $L$-functions than of Hasse--Weil
$L$-functions (though as fate would have it, Shimura--Taniyama is now
a theorem, while some cases of Artin's conjecture for $2$-dimensional
representations over $\Q$ remain open). However, one hypothesis in the
converse theorem emerged as a sticking point in the way of easily applying
it to Artin $L$-functions, namely the behavior under twist of the root
number in the functional equation, which is tantamount to proving the
existence of local root numbers for Artin representations. Langlands
\cite{langlands-artin} solved this problem by a direct (i.e.\ local) but
very involved computation. Shortly after, Deligne \cite{deligne} gave a
simpler global proof by the method of ``stability of $\epsilon$-factors''.

In this paper we show that the issue could have been
circumvented\footnote{That is not to say that it should have been. As
Langlands makes clear in his commentary \cite{langlands-comments2}
and \cite{langlands-comments1}, the existence of local root
numbers for Artin representations was an important confirmation of the
nascent local theory.},
in the sense that knowledge of the root number is not needed in the
converse theorem. We also incorporate the method of \cite{bk-classical}
to allow the non-trivial twists to have arbitrary poles. Precisely,
we show the following:
\begin{theorem}\label{thm1}
Let $\xi$ be a Dirichlet character modulo $N$, $k$ a positive integer
satisfying $\xi(-1)=(-1)^k$, and $\{\lambda_n\}_{n=1}^\infty$ a sequence
of complex numbers satisfying $\lambda_n=O(\sqrt{n})$ and the Hecke
relations, so that
\begin{equation}\label{eq:ep}
\sum_{n=1}^\infty\lambda_nn^{-s}=
\prod_p\frac1{1-\lambda_pp^{-s}+\xi(p)p^{-2s}},
\qquad\text{with }\overline{\lambda_p}=\overline{\xi(p)}\lambda_p
\text{ for each prime }p\nmid N.
\end{equation}

For any primitive Dirichlet character $\chi$ of conductor $q$ coprime to
$N$, define
$$
\Lambda_\chi(s)=\Gamma_\C\bigl(s+\tfrac{k-1}2\bigr)
\sum_{n=1}^\infty\lambda_n\chi(n)n^{-s}
$$
for $s\in\C$ with $\Re(s)>\frac32$, where
$\Gamma_\C(s)=2(2\pi)^{-s}\Gamma(s)$. Suppose, for every such $\chi$,
that $\Lambda_\chi(s)$ continues to a meromorphic function on $\C$ and
satisfies the functional equation
\begin{equation}\label{eq:fe}
\Lambda_\chi(s)=\epsilon_\chi(Nq^2)^{\frac12-s}
\overline{\Lambda_\chi(1-\bar{s})},
\end{equation}
for some $\epsilon_\chi\in\C$ (necessarily of magnitude $1$).
Let $\mathbf{1}$ denote the character of modulus $1$, and
suppose that there is a nonzero polynomial $P$ such that
$P(s)\Lambda_{\mathbf{1}}(s)$ continues to an entire function of finite
order.

Then one of the following holds:
\begin{itemize}
\item[(i)]$k=1$ and there are primitive characters
$\xi_1\pmod*{N_1}$ and $\xi_2\pmod*{N_2}$ such that
$N_1N_2=N$, $\xi_1\xi_2=\xi$ and
$\lambda_n=\sum_{d\mid n}\xi_1(n/d)\xi_2(d)$ for every $n$.
\item[(ii)]
$\sum_{n=1}^\infty\lambda_nn^{\frac{k-1}2}e^{2\pi inz}$
is a normalized Hecke eigenform in $S_k^{\rm new}(\Gamma_0(N),\xi)$.
\end{itemize}
\end{theorem}

The result can also be stated in representation-theoretic terms, as
follows.
\begin{theorem}\label{thm2}
Let $\A_\Q$ denote the ad\`ele ring of $\Q$, and
let $\pi=\pi_\infty\otimes\bigotimes_{v<\infty}\pi_v$ be an irreducible
admissible representation of $\GL_2(\A_\Q)$ with automorphic central
character and conductor $N$.
Assume that each $\pi_v$ is unitary and that $\pi_\infty$ is
a discrete series or limit of discrete series representation.
For each unitary id\`ele class character $\omega$ of conductor $q$
coprime to $N$, suppose that the complete $L$-functions
$$
\Lambda(s,\pi\otimes\omega)=\prod_vL(s,\pi_v\otimes\omega_v)
\quad\text{and}\quad
\Lambda(s,\widetilde{\pi}\otimes\omega^{-1})
=\prod_vL(s,\widetilde{\pi}_v\otimes\omega_v^{-1}),
$$
defined initially for $\Re(s)>\frac32$, continue to meromorphic
functions on $\C$ and satisfy a functional equation of the form
$$
\Lambda(s,\pi\otimes\omega)=\epsilon_\omega(Nq^2)^{\frac12-s}
\Lambda(1-s,\widetilde{\pi}\otimes\omega^{-1})
$$
for some complex number $\epsilon_\omega$.
Suppose also that there is a nonzero polynomial $P$
such that $P(s)\Lambda(s,\pi)$ continues to an entire function of
finite order. Then there is an automorphic representation
$\Pi=\bigotimes_v\Pi_v$ which is either cuspidal or an isobaric
sum of unitary id\`ele class characters, and satisfies
$\pi_v\cong\Pi_v$ for $v=\infty$
and every finite $v$ at which $\pi_v$ is unramified.
\end{theorem}

Many extensions and variations of the hypotheses of the two theorems
are possible. We mention a few:
\begin{enumerate}
\item The restriction to discrete series representations
was made for convenience and could be removed with more work.
It is also likely possible to formulate a version over number fields,
starting along the lines of \cite{bk1,bk2}.
\item The assumptions that $\lambda_n=O(\sqrt{n})$ in Theorem~\ref{thm1}
and that $\pi_v$ be unitary in Theorem~\ref{thm2} could be relaxed to
polynomial growth of the Satake parameters, at the expense of allowing
solutions corresponding to Eisenstein series of higher weight. Since we
have allowed the untwisted $L$-function to have finitely many poles,
that would include the Eisenstein series of weight $2$ and level $1$
(which is not modular), as in \cite{bk-classical}.
\item If we assume that the twisted $L$-functions are entire then,
using the method of \cite{diaconu}, it is enough to assume the
functional equation for the trivial character and characters of a single
well-chosen prime conductor $q$ (depending on $N$).
As shown in \cite[Theorem~2.5]{workshop}, the set of suitable $q$
has density $1$ in the set of all primes.
\item Our proof makes use of the Euler product, an idea that originates
with Conrey and Farmer \cite{cf}. It is not required in Weil's original
converse theorem, thanks to an abundance of twists, and one might ask
whether it is possible to eliminate both the root numbers and Euler
product. It is plausible that the answer is no, as the analogous question
for additive twists has a negative answer, as shown in \cite{steiner}.
However, adopting the philosophy espoused in \cite{workshop}, it is
likely possible to linearize the Euler product, replacing it by the
functional equation under twist by Ramanujan sums, $c_q(n)$.
\end{enumerate}
Finally, we note that stability of $\epsilon$-factors for more general
reductive groups and representations of their $L$-groups remains an
active area of research (see \cite{cst} for a recent survey), motivated
in part by applications involving converse theorems for $\GL_n$.
It would be interesting to understand the extent to which
our result can be generalized to higher rank.

\section{Lemmas}
We begin with a few preparatory lemmas. We assume the notation and
hypotheses of Theorem~\ref{thm1} throughout.

\begin{lemma}\label{lem:cq}
Let $q\nmid N$ be a prime number, and let
$$
c_q(n)=\sum_{\substack{a\pmod*{q}\\(a,q)=1}}e\!\left(\frac{an}{q}\right)
$$
be the associated Ramanujan sum, where $e(x)=e^{2\pi i x}$. Define
$$
\Lambda_{c_q}(s)=\Gamma_\C\bigl(s+\tfrac{k-1}2\bigr)
\sum_{n=1}^\infty\lambda_nc_q(n)n^{-s}.
$$
Then the ratio $D_q(s)=\Lambda_{c_q}(s)/\Lambda_{\mathbf{1}}(s)$ is a Dirichlet
polynomial satisfying the functional equation
$$
D_q(s)=\xi(q)q^{1-2s}\overline{D_q(1-\bar{s})}.
$$
\end{lemma}
\begin{proof}
This holds more generally for positive integers $q$ coprime to $N$, as
shown in \cite[Lemma~4.12]{workshop}. For completeness, we prove the claim
for prime values of $q$.
Let $\chi_0$ denote the trivial character mod $q$. Then a
straightforward calculation shows that
$c_q(n)=q-1-q\chi_0(n)$, so that
$$
\begin{aligned}
D_q(s)&=
\frac{\Lambda_{c_q}(s)}{\Lambda_{\mathbf{1}}(s)}
=q-1-\frac{q\sum_{n=1}^\infty\lambda_n\chi_0(n)n^{-s}}
{\sum_{n=1}^\infty\lambda_nn^{-s}}
=q-1-q(1-\lambda_qq^{-s}+\xi(q)q^{-2s})\\
&=-1+\lambda_qq^{1-s}-\xi(q)q^{1-2s},
\end{aligned}
$$
by \eqref{eq:ep}.
Hence
$$
\begin{aligned}
\xi(q)q^{1-2s}\overline{D_q(1-\bar{s})}
&=\xi(q)q^{1-2s}\bigl(-1+\overline{\lambda_q}
q^s-\overline{\xi(q)}q^{2s-1}\bigr)
=-1+\xi(q)\overline{\lambda_q}q^{1-s}-\xi(q)q^{1-2s}\\
&=D_q(s),
\end{aligned}
$$
since $\overline{\lambda_q}=\overline{\xi(q)}\lambda_q$,
by \eqref{eq:ep}.
\end{proof}

\begin{lemma}\label{lem:bk}
\hspace{1pt}
\begin{enumerate}
\item
$\Lambda_\chi(s)$ is entire of finite order for every primitive character
$\chi$ of prime conductor $q\nmid N$.
\item
$\Lambda_{\mathbf{1}}(s)$ is
entire apart from at most simple poles at $s\in\{0,1\}$ for $k=1$ and
$s\in\{-\frac12,\frac12,\frac32\}$ for $k=2$.
\end{enumerate}
\begin{proof}
These follow from the proof of \cite[Theorem~1.1]{bk-classical}.
Although the statement of that theorem includes a formula for the root number,
we verify that
no use of that hypothesis is made until \S3.1, where Weil's converse
theorem is applied. Thus we find that (1) holds and that
$\Lambda_{\mathbf{1}}(s)$ is entire apart from at most simple poles at
$s\in\{\frac{1\pm k}2\}$ for $k\ne2$ and
$s\in\{-\frac12,\frac12,\frac32\}$ for $k=2$.
Finally, the estimate $\lambda_n=O(\sqrt{n})$, together with the
functional equation \eqref{eq:fe}, rules out poles in the case $k>2$.
\end{proof}
\end{lemma}

\begin{lemma}\label{lem:nonvanishing}
For any prime $q\nmid N$ and any integer $a$,
there is a positive integer $n\equiv a\pmod*{q}$ such that
$\lambda_n\ne0$.
\end{lemma}
\begin{proof}
Suppose the conclusion is false for some $q$ and $a$. If $a\equiv0\pmod*{q}$
then we must have $\lambda_q=0$, but then the Euler product \eqref{eq:ep}
implies that $\lambda_{q^2}=-\xi(q)\ne0$. Hence we may assume that $(a,q)=1$.

Letting $\chi_0$
denote the trivial character mod $q$, we have
$\chi_0(n)=\frac{q-1-c_q(n)}{q}$, and thus
$$
\frac1{q}-\frac1{q(q-1)}c_q(n)+\frac1{q-1}
\sum_{\substack{\chi\pmod*{q}\\\chi\ne\chi_0}}\overline{\chi(a)}\chi(n)
$$
is the indicator function of the residue class of $a$. Hence, by
hypothesis we have
$$
\Lambda_{\mathbf{1}}(s)
=\frac1{q-1}\Lambda_{c_q}(s)-\frac{q}{q-1}
\sum_{\substack{\chi\pmod*{q}\\\chi\ne\chi_0}}\overline{\chi(a)}
\Lambda_\chi(s).
$$
Applying the functional equation and making use of Lemma~\ref{lem:cq},
this implies that
$$
\epsilon_{\mathbf{1}}N^{\frac12-s}\overline{\Lambda_{\mathbf{1}}(1-\bar{s})}
=\frac{\epsilon_{\mathbf{1}}\xi(q)}{q-1}(Nq^2)^{\frac12-s}
\overline{\Lambda_{c_q}(1-\bar{s})}
-\frac{q}{q-1}(Nq^2)^{\frac12-s}
\sum_{\substack{\chi\pmod*{q}\\\chi\ne\chi_0}}\overline{\chi(a)}
\epsilon_\chi\overline{\Lambda_\chi(1-\bar{s})}.
$$
Multiplying both sides by $\overline{\epsilon_{\mathbf{1}}}(Nq^2)^{s-\frac12}$,
replacing $s$ by $1-\bar{s}$ and conjugating, we obtain
$$
q^{1-2s}\Lambda_{\mathbf{1}}(s)
=\frac{\overline{\xi(q)}}{q-1}\Lambda_{c_q}(s)
-\frac{q}{q-1}
\sum_{\substack{\chi\pmod*{q}\\\chi\ne\chi_0}}\chi(a)
\epsilon_{\mathbf{1}}\overline{\epsilon_\chi}
\Lambda_\chi(s).
$$
Comparing the Dirichlet coefficients of both sides at $q$, 
we see that $\lambda_q=0$. In turn, as above, this implies that
$\lambda_{q^2}=-\xi(q)$. Comparing coefficients at $q^2$, we thus have
$q=-1$, which is absurd. This concludes the proof.
\end{proof}

In \cite[Theorem~9]{li}, Li proved that a cuspform whose $L$-function
satisfies both the Euler product \eqref{eq:ep} and the functional equation
\eqref{eq:fe} for $\chi=\mathbf{1}$ must be primitive. Our final result
of this section constitutes an extension of Li's result that includes
the Eisenstein series.
\begin{lemma}\label{lem:newform}
Let $f\in M_k(\Gamma_0(N),\xi)$ have Fourier expansion
$\sum_{n=0}^\infty f_ne(nz)$, and assume that it is a
normalized eigenfunction for the full Hecke algebra, so that
$$
\sum_{n=1}^\infty f_nn^{-s-\frac{k-1}2}
=\prod_p\frac1{1-f_pp^{-s-\frac{k-1}2}+\xi(p)p^{-2s}}.
$$
Let
$$
\Lambda_f(s)=\Gamma_\C\bigl(s+\tfrac{k-1}2\bigr)
\sum_{n=1}^\infty f_nn^{-s-\frac{k-1}2}
$$
be the associated complete $L$-function, and assume that it
satisfies the functional equation
$$
\Lambda_f(s)=\epsilon N^{\frac12-s}\overline{\Lambda_f(1-\bar{s})}
$$
for some $\epsilon\in\C$.
Then one of the following holds:
\begin{itemize}
\item[(i)]$f$ is a primitive cuspform, i.e.\ a normalized Hecke eigenform in
$S_k^{\rm new}(\Gamma_0(N),\xi)$.
\item[(ii)]There are Dirichlet characters $\xi_1\pmod*{N_1}$ and
$\xi_2\pmod*{N_2}$ such that $\xi_1$ is primitive,
$N_1N_2=N$, $\xi_1\xi_2=\xi$ and
$$
f_n=\sum_{d\mid n}\xi_1(n/d)\xi_2(d)d^{k-1}
$$
for every $n>0$. If $k\ne2$ then $\xi_2$ is primitive. If $k=2$ then
$\xi_2$ need not be primitive (and in fact it must be imprimitive if
$\xi_1$ and $\xi_2$ are both trivial), but if $N_2^*$ denotes the conductor of
$\xi_2$ then $N_2/N_2^*$ is squarefree and $(N_2/N_2^*,N_2^*)=1$.
\end{itemize}
\end{lemma}
\begin{proof}[Proof (sketch)]
Let $X$ denote the set of pairs $(\xi_1,\xi_2)$, where
$\xi_1\pmod*{N_1}$ and $\xi_2\pmod*{N_2}$ are primitive Dirichlet
characters such that $N_1N_2\mid N$, $\xi_1(-1)\xi_2(-1)=(-1)^k$ and
if $k=1$ then $\xi_1(-1)=1$.
To any pair $(\xi_1,\xi_2)\in X$ we associate
the $L$-series
$$
L_{\xi_1,\xi_2}(s)=L(s+\tfrac{k-1}2,\xi_1)L(s-\tfrac{k-1}2,\xi_2),
$$
where the factors on the right-hand side are the usual Dirichlet $L$-functions.

Next let $C$ denote the set of all primitive weight-$k$ cuspforms $g$
of conductor dividing $N$.
To $g\in C$ with Fourier expansion $\sum_{n=1}^\infty g_ne(nz)$
we associate the $L$-series
$$
L_g(s)=\sum_{n=1}^\infty g_nn^{-s-\frac{k-1}2}.
$$

Let $L_f(s)=\sum_{n=1}^\infty f_nn^{-s-\frac{k-1}2}$ denote the finite
$L$-series of $f$. Then by newform theory and the
description of Eisenstein series in \cite[\S4.7]{miyake},
there are Dirichlet polynomials $D_{\xi_1,\xi_2}$ and $D_g$ such that
\begin{equation}\label{eq:linear}
L_f(s)=\sum_{(\xi_1,\xi_2)\in X}D_{\xi_1,\xi_2}(s)L_{\xi_1,\xi_2}(s)
+\sum_{g\in C}D_g(s)L_g(s).
\end{equation}
Further, the coefficients of each Dirichlet polynomial are supported on
divisors of $N$.

Following \cite{KMP}, we will say that Euler products $L_1(s)$ and
$L_2(s)$ are \emph{equivalent} if their Euler factors agree for all but at
most finitely many primes, and \emph{inequivalent} otherwise. It follows
from the Rankin--Selberg method (see, e.g., \cite[Corollary~4.4]{bk-gln})
that the elements of
$\{L_{\xi_1,\xi_2}:(\xi_1,\xi_2)\in X\}\cup\{L_g:g\in C\}$
are pairwise inequivalent.  Combining this with \cite[Theorem~2]{KMP},
we see that the right-hand side of \eqref{eq:linear} has exactly one
nonzero term.  If the nonzero term corresponds to a cuspform $g\in C$
then $f$ is also cuspidal, and thus the conclusion follows from Li's
theorem \cite[Theorem~9]{li}.

Hence we may suppose that
$L_f(s)=D_{\xi_1,\xi_2}(s)L_{\xi_1,\xi_2}(s)$
for some pair $(\xi_1,\xi_2)\in X$. Since the function
$\Lambda_{\xi_1,\xi_2}(s)=\Gamma_\C(s+\frac{k-1}2)L_{\xi_1,\xi_2}(s)$
satisfies a functional equation of level $N_1N_2$,
$D_{\xi_1,\xi_2}(s)=\Lambda_f(s)/\Lambda_{\xi_1,\xi_2}(s)$
satisfies a functional equation of level $N/N_1N_2$, i.e.\
\begin{equation}\label{eq:Dfe}
D_{\xi_1,\xi_2}(s)=\epsilon_{\xi_1,\xi_2}
\left(\frac{N}{N_1N_2}\right)^{\frac12-s}
\overline{D_{\xi_1,\xi_2}(1-\bar{s})}
\end{equation}
for a suitable constant $\epsilon_{\xi_1,\xi_2}$.
On the other hand, since the coefficients of $D_{\xi_1,\xi_2}(s)$
are supported on divisors of $N$, from the Euler products for $L_f(s)$ and
$L_{\xi_1,\xi_2}(s)$ we have
$$
D_{\xi_1,\xi_2}(s)=\prod_{p\mid N}
\frac{(1-\xi_1(p)p^{-s-\frac{k-1}2})(1-\xi_2(p)p^{-s+\frac{k-1}2})}
{1-f_pp^{-s-\frac{k-1}2}}.
$$
This ratio must be entire, and by the functional equation \eqref{eq:Dfe},
its zeros are symmetric with respect to the line $\Re(s)=\frac12$.
By the $\Q$-linear independence of $\log{p}$ for primes $p$,
the same is true of each individual Euler factor.

Now, if $k\ne2$ then a straightforward case-by-case analysis shows that
this is only possible if
$$
\frac{(1-\xi_1(p)p^{-s-\frac{k-1}2})(1-\xi_2(p)p^{-s+\frac{k-1}2})}
{1-f_pp^{-s-\frac{k-1}2}}=1
$$
for each $p$, so that $D_{\xi_1,\xi_2}(s)=1$.
Thus, $L_f(s)=L_{\xi_1,\xi_2}(s)$,
and by the functional equation \eqref{eq:Dfe} we have
$N=N_1N_2$. This yields the desired conclusion for $k\ne2$.

If $k=2$ then, since the zeros of
$1-\xi_2(p)p^{-s+\frac12}$ lie on the line $\Re(s)=\frac12$,
there are two possibilities for each $p$:
\begin{itemize}
\item[(i)]
$\displaystyle
\frac{(1-\xi_1(p)p^{-s-\frac12})(1-\xi_2(p)p^{-s+\frac12})}
{1-f_pp^{-s-\frac12}}=1;
$
\item[(ii)]
$\displaystyle
\xi_2(p)\ne0\quad\text{and}\quad
\frac{(1-\xi_1(p)p^{-s-\frac12})(1-\xi_2(p)p^{-s+\frac12})}
{1-f_pp^{-s-\frac12}}=1-\xi_2(p)p^{-s+\frac12}.
$
\end{itemize}
Let $S$ denote the set of $p\mid N$ for which case (ii) applies.
Then we have
$$
D_{\xi_1,\xi_2}(s)=\prod_{p\in S}(1-\xi_2(p)p^{-s+\frac12}).
$$
For each $p\in S$, note that
$1-\xi_2(p)p^{-s+\frac12}$ satisfies a functional equation
of level $p$. Comparing with \eqref{eq:Dfe}, we see that
$N/N_1N_2=\prod_{p\in S}p$. Moreover, since $\xi_2(p)\ne0$ for each
$p\in S$, we have $(N/N_1N_2,N_2)=1$.
Replacing $\xi_2$ by the character of modulus $N/N_1$ that it induces,
we get the conclusion of the lemma.
\end{proof}

\section{Proof of Theorem~\ref{thm1}}
We first apply Lemma~\ref{lem:bk} to constrain the poles of
$\Lambda_{\mathbf{1}}(s)$ and $\Lambda_\chi(s)$ for primitive characters
$\chi$ of prime conductor $q\nmid N$. When $k\le2$ we suppose for now
that $\Lambda_{\mathbf{1}}(s)$ is entire, and return to the general
case below. Thus, both $\Lambda_{\mathbf{1}}(s)$ and $\Lambda_\chi(s)$
are entire of finite order. By the Phragm\'en--Lindel\"of convexity
principle, they are bounded in vertical strips.

Let $\H=\{z\in\C:\Im(z)>0\}$ denote the upper half-plane.
For $z\in\H$, set
$$
f_n=\lambda_nn^{\frac{k-1}2},
\quad
f(z)=\sum_{n=1}^\infty f_ne(nz)
\quad\text{and}\quad
\bar{f}(z)=\sum_{n=1}^\infty\overline{f_n}e(nz).
$$
For any function $g:\H\to\C$ and any
matrix
$\gamma=\begin{psmallmatrix}a&b\\c&d\end{psmallmatrix}\in\GL_2(\R)$
of positive determinant, let $g|\gamma$ denote the function
$$
(g|\gamma)(z)=(\det\gamma)^{k/2}(cz+d)^{-k}g\!\left(\frac{az+b}{cz+d}\right).
$$
Then, by Hecke's argument \cite[Theorem~4.3.5]{miyake},
the functional equation \eqref{eq:fe} for
$\chi=\mathbf{1}$ implies that
$f|\begin{psmallmatrix}&-1\\N&\end{psmallmatrix}=
i^k\epsilon_{\mathbf{1}}\bar{f}$.
Note that
$$
\begin{pmatrix}1&\\N&1\end{pmatrix}
=\begin{pmatrix}&-1\\N\end{pmatrix}
\begin{pmatrix}1&1\\&1\end{pmatrix}^{-1}
\begin{pmatrix}&-1\\N\end{pmatrix}^{-1}.
$$
Since $f$ and $\bar{f}$ are Fourier series, it follows that
$f|\begin{psmallmatrix}1&1\\&1\end{psmallmatrix}=f$
and
$f|\begin{psmallmatrix}1&\\N&1\end{psmallmatrix}=f$.

If $\gamma,\gamma'\in\Gamma_0(N)$ have the same top row, then a
computation shows that $\gamma'\gamma^{-1}$ is a power of
$\begin{psmallmatrix}1&\\N&1\end{psmallmatrix}$, so that
$f|\gamma'=f|\gamma$. Thus, $f|\gamma$ depends only on the top row of
$\gamma$. With this in mind, we will write $\gamma_{q,a}$ to denote any
element of $\Gamma_0(N)$ with top row
$\begin{psmallmatrix}q&-a\end{psmallmatrix}$.

Let $q$ be a prime not dividing $N$, and let $\chi$ be a character
modulo $q$, not necessarily primitive. Define
$$
f_\chi=\sum_{\substack{a\pmod*{q}\\(a,q)=1}}\overline{\chi(a)}
f\left|\begin{pmatrix}1&a/q\\&1\end{pmatrix}\right.
\quad\text{and}\quad
\bar{f}_{\overline{\chi}}=\sum_{\substack{a\pmod*{q}\\(a,q)=1}}\chi(a)
\bar{f}\left|\begin{pmatrix}1&a/q\\&1\end{pmatrix}\right..
$$
Substituting the Fourier expansion for $f$, we see that $f_\chi$ has a
Fourier expansion with coefficients
$$
f_n\sum_{\substack{a\pmod*{q}\\(a,q)=1}}
\overline{\chi(a)}e\!\left(\frac{an}{q}\right)
=f_n\begin{cases}
c_q(n)&\text{if }\chi\text{ is trivial},\\
\tau(\overline{\chi})\chi(n)&\text{otherwise},
\end{cases}
$$
and similarly for $\bar{f}_{\overline{\chi}}$.
Set
$$
C_\chi=\begin{cases}
\overline{\xi(q)}&\text{if }\chi\text{ is trivial},\\
\chi(-N)\epsilon_{\mathbf{1}}\overline{\epsilon_\chi
\tau(\overline{\chi})/\tau(\chi)}&\text{otherwise}.
\end{cases}
$$
Then, by \eqref{eq:fe}, Lemma~\ref{lem:cq} and Hecke's argument, we have
$f_\chi|\begin{psmallmatrix}&-1\\Nq^2&\end{psmallmatrix}=
i^k\chi(-N)\epsilon_{\mathbf{1}}\overline{C_\chi}\bar{f}_{\overline{\chi}}$.

Suppose that $a$ and $m$ are integers satisfying
$Nam\equiv-1\pmod*{q}$. Then
$$
\gamma_{q,a}=q
\begin{pmatrix}&-1\\N&\end{pmatrix}
\begin{pmatrix}1&m/q\\&1\end{pmatrix}
\begin{pmatrix}&-1\\Nq^2&\end{pmatrix}^{-1}
\begin{pmatrix}1&a/q\\&1\end{pmatrix}^{-1}
=\begin{pmatrix}q&-a\\-Nm&\frac{Nam+1}q\end{pmatrix}
$$
is an element of $\Gamma_0(N)$ with top row
$\begin{psmallmatrix}q&-a\end{psmallmatrix}$.
Thus, we have
\begin{equation}\label{eq:main}
\begin{aligned}
\sum_{\substack{a\pmod*{q}\\(a,q)=1}}C_\chi\overline{\chi(a)}
f\left|\begin{pmatrix}1&a/q\\&1\end{pmatrix}\right.
&=C_\chi f_\chi=i^k\chi(-N)\epsilon_{\mathbf{1}}\bar{f}_{\overline{\chi}}\left|
\begin{pmatrix}&-1\\Nq^2&\end{pmatrix}^{-1}\right.\\
&=i^k\epsilon_{\mathbf{1}}
\sum_{\substack{m\pmod*{q}\\(m,q)=1}}\chi(-Nm)
\bar{f}\left|\begin{pmatrix}1&m/q\\&1\end{pmatrix}
\begin{pmatrix}&-1\\Nq^2&\end{pmatrix}^{-1}\right.\\
&=\sum_{\substack{m\pmod*{q}\\(m,q)=1}}\chi(-Nm)
f\left|\begin{pmatrix}&-1\\N&\end{pmatrix}
\begin{pmatrix}1&m/q\\&1\end{pmatrix}
\begin{pmatrix}&-1\\Nq^2&\end{pmatrix}^{-1}\right.\\
&=\sum_{\substack{a\pmod*{q}\\(a,q)=1}}\overline{\chi(a)}
f\left|\gamma_{q,a}
\begin{pmatrix}1&a/q\\&1\end{pmatrix}\right..
\end{aligned}
\end{equation}

Fix a residue $b$ coprime to $q$. Multiplying both sides by
$\chi(b)$ and averaging over $\chi$, we obtain
$$
f\left|\gamma_{q,b}\begin{pmatrix}1&b/q\\&1\end{pmatrix}\right.
=\frac1{\varphi(q)}\sum_{\chi\pmod*{q}}\chi(b)
\sum_{\substack{a\pmod*{q}\\(a,q)=1}}C_\chi\overline{\chi(a)}
f\left|\begin{pmatrix}1&a/q\\&1\end{pmatrix}\right..
$$
Replacing $a$ by $ab$ on the right-hand side, writing
$$
\widehat{C}_q(a)=\frac1{\varphi(q)}\sum_{\chi\pmod*{q}}C_\chi\overline{\chi(a)}
$$
and applying $\begin{psmallmatrix}1&-b/q\\&1\end{psmallmatrix}$
on the right, we obtain
$$
f|\gamma_{q,b}=\sum_{\substack{a\pmod*{q}\\(a,q)=1}}\widehat{C}_q(a)
f\left|\begin{pmatrix}1&(a-1)b/q\\&1\end{pmatrix}\right..
$$
From this we see that $f|\gamma_{q,b}$
has a Fourier expansion, with Fourier coefficients $f_nS_q(bn)$, where
$$
S_q(x)=\sum_{\substack{a\pmod*{q}\\(a,q)=1}}\widehat{C}_q(a)
e\!\left(\frac{(a-1)x}{q}\right).
$$

Now, let $\gamma=\begin{psmallmatrix}q&-b\\-Nm&r\end{psmallmatrix}$ be an
arbitrary element of $\Gamma_1(N)$. If $m=0$ then $\gamma$ is (up to sign,
if $N\le2$) a power of $\begin{psmallmatrix}1&1\\&1\end{psmallmatrix}$,
so that $f|\gamma=f$. Otherwise, multiplying $\gamma$ on the left
by $\begin{psmallmatrix}1&1\\&1\end{psmallmatrix}^{-j}$ leaves
$f|\gamma$ unchanged and replaces $q$ by $q+jmN$. By Dirichlet's
theorem, we may assume without loss of generality that $q$ is prime.
Since $q\equiv1\pmod{N}$, we have
$$
\begin{pmatrix}q&-1\\1-q&1\end{pmatrix}
=\begin{pmatrix}1&1\\&1\end{pmatrix}^{-1}
\begin{pmatrix}1&\\N&1\end{pmatrix}^{\frac{1-q}{N}},
$$
so that $f|\gamma_{q,1}=f$.
Given any residue $x\pmod*{q}$, by Lemma~\ref{lem:cq} we may choose
$n$ such that $n\equiv x\pmod*{q}$ and $f_n\ne0$. Equating Fourier
coefficients of $f|\gamma_{q,1}$ and $f$, it follows that $S_q(x)=1$.
In turn, this implies that $f|\gamma_{q,b}=f$, and
thus $f|\gamma=f$ for all $\gamma\in\Gamma_1(N)$.

Next consider an arbitrary $\gamma=\gamma_{q,b}\in\Gamma_0(N)$.
As above, we may assume that $q$ is prime.
Moreover, for any $a$ coprime to $q$, we have
$\gamma_{q,a}\gamma^{-1}\in\Gamma_1(N)$,
so that $f|\gamma_{q,a}=f|\gamma$.
Taking $\chi$ equal to the trivial character mod $q$ in
\eqref{eq:main}, we thus find that
\begin{equation}\label{eq:fgamma}
\sum_{\substack{a\pmod*{q}\\(a,q)=1}}
\bigl(f|\gamma-\overline{\xi(q)}f\bigr)
\left|\begin{pmatrix}1&a/q\\&1\end{pmatrix}\right.
=0.
\end{equation}
We showed above that $f|\gamma$ has a Fourier expansion.
Writing $a_n$ for the Fourier coefficients, \eqref{eq:fgamma} implies
that $(a_n-\overline{\xi(q)}f_n)c_q(n)=0$ for every $n$.
Since $c_q(n)$ never vanishes, we have $a_n=\overline{\xi(q)}f_n$,
so that $f|\gamma=\overline{\xi(q)}f$.
Thus, we have shown that $f\in M_k(\Gamma_0(N),\xi)$.

Next, by Lemma~\ref{lem:newform}, either $f$ is a primitive
cuspform or there are Dirichlet characters $\xi_1\pmod*{N_1}$
and $\xi_2\pmod*{N_2}$ such that $N_1N_2=N$, $\xi_1\xi_2=\xi$ and
\begin{equation}\label{eq:eisenstein}
f_n=\sum_{d\mid n}\xi_1(n/d)\xi_2(d)d^{k-1}
\end{equation}
for every $n$.
For $k\ge2$, we consider \eqref{eq:eisenstein} at $n=q_1\cdots q_m$,
where $q_1,\ldots,q_m$ are the $m$ smallest primes $\equiv1\pmod*{N}$.
For this $n$ we see that
$$
\lambda_n=f_nn^{-\frac{k-1}2}
=\prod_{i=1}^m\Bigl(q_i^{\frac{k-1}2}+q_i^{-\frac{k-1}2}\Bigr)
\ge\prod_{i=1}^m\Bigl(q_i^{\frac12}+q_i^{-\frac12}\Bigr),
$$
so that
$$
\frac{\lambda_n}{\sqrt{n}}\ge\prod_{i=1}^m\bigl(1+q_i^{-1}\bigr).
$$
By Dirichlet's theorem, the right-hand side grows without bound as
$m\to\infty$. This contradicts the hypothesis that $\lambda_n=O(\sqrt{n})$,
so $f$ must be a primitive cusp form. When $k=1$, $f$ need not be
cuspidal, but in this case Lemma~\ref{lem:newform} asserts that $\xi_1$
and $\xi_2$ are primitive. Thus, we have verified the conclusion of
Theorem~\ref{thm1}.

It remains only to handle the possibility that $\Lambda_{\mathbf{1}}(s)$
has poles when $k\le2$. In this case we fix an odd prime $q\nmid N$
and a primitive character $\chi\pmod*{q}$, and consider the sequence
$\lambda_n'=\lambda_n\chi(n)$ in place of $\lambda_n$,
$\xi\chi^2$ in place of $\xi$ and $Nq^2$ in place of $N$.
Then all of the hypotheses
of Theorem~\ref{thm1} are satisfied for these data,
and the associated $L$-function $\Lambda_{\mathbf{1}}(s)$
is entire.  Thus,
by what we have already shown, either there is a primitive cuspform
$f'\in S_k^{\rm new}(\Gamma_0(Nq^2),\xi\chi^2)$ with Fourier
coefficients $\lambda_n'n^{\frac{k-1}2}$, or $k=1$ and there are primitive
characters $\xi_1'$ and $\xi_2'$ such that
$\lambda_n'=\sum_{d\mid n}\xi_1'(n/d)\xi_2'(d)$.

Consider the cuspidal case first. By newform theory
\cite[Theorem~3.2]{atkin-li}, we can twist $f'$ by $\overline{\chi}$,
i.e.\ there is a primitive cuspform $f$ of conductor
$Nq^j$ for some $j$, with Fourier coefficients
$\lambda_n'\overline{\chi(n)}n^{\frac{k-1}2}=\lambda_nn^{\frac{k-1}2}$
for every $n$
coprime to $q$. Since $q$ was arbitrary, we can apply this argument to two
different choices of $q$. Then strong multiplicity one implies that $f$ has
conductor $N$ and Fourier coefficients $\lambda_nn^{\frac{k-1}2}$ for every
$n$, as desired.

In the non-cuspidal case, let $\xi_i\pmod{N_i}$ ($i=1,2$)
be the primitive character inducing $\xi_i'\overline{\chi}$.  Then as
above we find that $\lambda_n=\sum_{d\mid n}\xi_1(n/d)\xi_2(d)$ for all
$n$ coprime to $q$. In particular,
\begin{equation}\label{eq:k1}
\lambda_p=\xi_1(p)+\xi_2(p)
\quad\text{for all sufficiently large primes }p.
\end{equation}
Note that $\xi_1$ and $\xi_2$ have opposite parity. If we normalize
$\xi_1$ to be even then, since $\xi_1$ and $\xi_2$ are primitive,
Dirichlet's theorem implies that they are uniquely determined by
\eqref{eq:k1}. Hence, using two choices for $q$,
we see that $N_1N_2=N$ and $\lambda_n=\sum_{d\mid n}\xi_1(n/d)\xi_2(d)$
for all $n$. This completes the proof.

\bibliographystyle{amsplain}
\bibliography{rootnumber}
\end{document}